\documentclass[11pt,leqno,a4paper, english]{smfart}

\usepackage[]{hyperref}
\hypersetup{
    colorlinks=true,       
    linkcolor=red,          
    citecolor=blue,        
    filecolor=magenta,      
    urlcolor=cyan           
}

\usepackage{amsmath}
\usepackage{amsfonts,amssymb}
\usepackage{enumerate}
\usepackage{mathrsfs}
\usepackage{amsthm}
\usepackage{a4wide}

\theoremstyle{plain}
\newtheorem{theorem}{Theorem }[section]
\newtheorem{prop}[theorem]{Proposition}
\newtheorem{lemm}[theorem]{Lemma}

\theoremstyle{definition}
\newtheorem{rema}[theorem]{Remark}

\DeclareSymbolFont{pletters}{OT1}{cmr}{m}{sl}
\DeclareMathSymbol{s}{\mathalpha}{pletters}{`s}

\def\L#1{\langle #1 \rangle}

\def\mez{\frac{1}{2}}

\def\xR{\mathbf{R}}

\numberwithin{equation}{section}

\title{Laplace eigenfunctions  and  damped wave equation~II: product manifolds. }
\author{N.Burq}
\address{Laboratoire de Math\'ematiques UMR 8628 du CNRS. Universit\'e Paris-Sud, B\^atiment 425, 91405 Orsay Cedex}
 \author{C.Zuily }
 \address{Laboratoire de Math\'ematiques UMR 8628 du CNRS. Universit\'e Paris-Sud, B\^atiment 425, 91405 Orsay Cedex}
 \thanks{N.B. was supported in part by Agence Nationale de la Recherche
  project NOSEVOL, 2011 BS01019 01. N.B. and C. Z. were supported in part by Agence Nationale de la Recherche
  project  ANA\'E ANR-13-BS01-0010-03.}
\begin{abstract}
The purpose of this article is to study possible concentrations of eigenfunctions of Laplace operators (or more generally quasi-modes) on product manifolds. We show that the approach of the first author and Zworski~\cite{BuZw03, BuZw03-1} applies (modulo rescalling) and  deduce new stabilization results for weakly damped wave equations which extend to product manifolds previous results by Leautaud-Lerner~\cite{LeLe14} obtained for products of tori.
\end{abstract}
\begin{altabstract}
Dans cet article, on \'etudie les concentrations possibles des fonctions propres du Laplacien  (ou plus g\'en\'eralement de quasi-modes) sur des vari\'et\'es produit. On d\'emontre que l'approche du premier auteur avec M. Zworski~\cite{BuZw03, BuZw03-1} s'applique (modulo un changement d'\'echelle) et on en d\'eduit de nouveaux r\'esultats de stabilization pour l'\'equation des ondes faiblement amortie qui g\'en\'eralisent au cas des vari\'et\'es produits des r\'esultats ant\'erieurs de Leautaud-Lerner~\cite{LeLe14} obtenus dans le cas de produits de tores.
\end{altabstract}
\date{\empty}
\begin{document}
 \maketitle
 \section{Notations and main results}
   In this work we continue our investigation \cite{BuZu15} of concentration properties of eigenfunctions (or more generally quasimodes) of the Laplace-Beltrami  operator  on submanifolds and we study here the very particular setting of product manifolds.
   
        Let  $(M_j, g_j), j=1,2$ be two compact manifolds. We denote by  $(M= M_1\times M_2,g= g_1\otimes g_2)$ the product, and  by $d_j$ (resp. $d$) the geodesic distance  in $M_j$ (resp. $M$). Let $q_0 \in M_2$ and
    $$ \Sigma = M_1 \times \{q_0\}.$$
  For $\beta>0$ we introduce
  \begin{equation}\label{Nbeta}
  N_\beta = \{m=(p,q) \in M: d(m, \Sigma) <\beta\} = M_1 \times \{q\in M_2: d_2(q,q_0)<\beta\}.
\end{equation}
   Our first result is the following.
         
     \begin{theorem}\label{quasi-mode}
       For any $\delta>0$, there exists $C>0, h_0>0$ such that for every $0<h\leq h_0$ and every  solution $\psi  \in H^2(M)$ of the equation  
$$(h^2 \Delta_g   +1)\psi  = F$$
we have the estimate
\begin{equation}\label{qm-N}
\Vert \psi  \Vert_{L^2(N_{h^{\delta}})} \leq C\big( \Vert \psi  \Vert_{L^2(N_{2h^{\delta} }  \setminus N_{h^{\delta}})} + h^{2\delta -2} \Vert F \Vert_{L^2(N_{2h^{\delta} } )}\big).
 \end{equation} \end{theorem}
 
 As an application of Theorem~\ref{quasi-mode}, we consider  weakly damped wave equations on a compact Riemaniann manifold $(\mathcal{M},g),$ 
  \begin{equation}\label{damped}
   (\partial_t^2 - \Delta_g+ b(m) \partial_t )u =0,\quad  (u, \partial_tu ) \arrowvert_{t=0} = ( u_0, u_1)\in H^{1+k}(\mathcal{M}) \times H^k( \mathcal{M} ), 
   \end{equation}
where 
$0\leq b \in L^\infty (\mathcal{M})$, for which the energy
$$ E(u)( t) = \int_{\mathcal{M}} \big(g_p(\nabla_g u(t,m), \nabla_g u(t,m)) + | \partial_t u(t,m)|^2 \big)dv_g(m)
$$
is decaying since  $\frac d{dt} E(u)(t) = - \int_\mathcal{M} b(m)|\partial_tu(t,m)|^2 dv_g(m) \leq 0.$ 
Let 
$$\omega= \cup\{ U \text{ open }:  \mathrm{ess} \inf_{U} b >0\}
$$  be the domain where effective damping occurs. We denote by 
$$\mathcal{GC}=\{ \rho\in S^{\star}\mathcal{M} : \exists s\in \mathbb{R}; \Phi(s) \rho =(m_1, \xi_1)  \in S^{\star}\omega\},$$
the (open) set of geometrically controlled points (here $\Phi(s)$ is the bicharacteristic flow). Let 
\begin{equation}\label{T}
 \mathcal{T}= S^*\mathcal{M} \setminus \mathcal{GC},\quad T= \Pi_x \mathcal{T}
\end{equation}
where $\mathcal{T}$ is the trapped set  and  $\Pi_x$ the projection on the base manifold $\mathcal{M}$. 

Our second result is the following.
 \begin{theorem}\label{stabi}
Assume that
\begin{enumerate} 

 \item there exists a neighborhood $V$ of $T$ in $\mathcal{M},$  a compact   Lipschitz Riemannian manifold $(M_1, g_1)$   of dimension $k $ and a Lipschitz  isometry
 $$\Theta: V \to (M_1 \times B(0,1), \widetilde{g}= g_1\otimes g_2)$$
   
   where $B(0,1)$ is the unit ball in $\mathbb{R}^{d-k}$ endowed with the (Lipschitz) metric $g_2$,
        \item  there exists $\gamma>0,c,C>0$ such that
\begin{equation}\label{below}
      c |z|^{2 \gamma} \leq b(\Theta^{-1}(p,z)) \leq C |z|^{2 \gamma},\quad  \forall (p,z)\in M_1 \times B(0,1).
  \end{equation}
  \end{enumerate}
Then there exists $C>0$ such that for any $(u_0, u_1) \in H^{2} (\mathcal{M}) \times H^1(\mathcal{M})$, the solution $u$ to~\eqref{damped} satisfies 
$$
E(u)^{1/2} (t) \leq \frac {C} { t^{ 1+ \frac 1 \gamma }} \bigl( \|  u_0\|_{H^ 2(\mathcal{M})} + \| u_1 \|_{H^1(\mathcal{M})}\bigr).
$$ 
\end{theorem}

\begin{rema}
 A simpler (but weaker) statement would be to assume
 \begin{align*} 
 &(i)\quad (\mathcal{M},g) = (M_1 \times M_2, g_1 \otimes g_2), \quad  q_0 \in M_2,\quad  T= \Sigma = M_1 \times \{q_0\}, \\
   &(ii) \quad    cd(m, \Sigma)^{2 \gamma} \leq b(m) \leq Cd(m, \Sigma)^{2 \gamma}  \quad  \forall m \in M_1 \times U.
 \end{align*}
\end{rema}

 It is classical that for non trivial  dampings $b\geq 0$, the energy of  solution to~\eqref{damped} converge to $0$ as $t$ tend to infinity.  The rate of decay is {\em uniform} (and hence exponential) in energy space if and only if the {\em geometric control condition}~\cite{BaLeRa92, BuGe96} is satisfied. In~\cite{BuZu15}, we explored the question when some trajectories are trapped and exhibited decay rates (assuming more regularity on the initial data).  This latter question was previously studied in a general setting in~\cite{Le96} and on tori in ~\cite{BuHi05, Ph07, AL14} (see also~\cite{BuZw03, BuZw03-1}) and more recently by Leautaud-Lerner~\cite{LeLe14}. The geometric assumptions in~\cite{BuZu15} are much more general than in~\cite{LeLe14} which is essentially restricted to the case of product of flat tori. On the other hand, due to this more favorable geometry, the decay rate in~\cite{LeLe14} is better than in~\cite{BuZu15}. Theorem~\ref{stabi} shows  that Leautaud-Lerner's result (the better decay rate) extends straightforwardly  to the case of product manifolds $(M_1\times M_2, g = g_1 \otimes g_2)$.
 \begin{rema}
 \begin{enumerate}
 \item According to Theorem 1.6 in \cite{LeLe14} the rate of decay in $t$ obtained in Theorem \ref{stabi} above is optimal in general.
 \item Theorem~\ref{quasi-mode} is a propagation result in  the $z$-variable in $B(0,1)$, and since $z$ is actually very close to $0$, the relevant object is $g_2(0)$ (constant coefficients) rather than $g_2(z)$. 
 The smaller $\delta$, the further we need to propagate in the $z$ variable and hence we better the  quasi modes we need to consider (due to the worse error factor $h^{2\delta -2}$)
\item The case $\delta = 1/2$ in Theorem~\ref{quasi-mode} is a particular case of our results in~\cite{BuZu15} (which are actually much more general and hold without the "product" assumption on the geometry). On the other hand, the results in~\cite{BuZu15} are {\em local}, while for  $\delta <1/2$, the estimate~\eqref{qm-N} is {\em non local}. Indeed, trying to replace $\psi$ by $\chi \psi $ will add to the r.h.s. a term ($[h^2 \Delta, \chi]\psi$) which is clearly bounded in $L^2$ by $O(h)$, giving an error of order $O(h^{2\delta -1})\gg 1$ to the final result. On the other hand, as soon as $\delta<1/2$, estimate~\eqref{qm-N} is false without the product structure assumption as can be easily seen on spheres by considering the eigenfunctions $e_n = (x_1 + i x_2)^n $ with eigenvalues $\lambda_n = n(n+d-1)= h_n^{-2}$ which concentrate in an $h_n^{1/2}$-neighborhood of the equator
$$E=  \{ x\in \mathbb{R}^{d+1}: |x| =1,  x_3=\cdots = x_{d+1} =0\}.$$
In this case, we get  $(h_n^2 \Delta +1 ) e_n =0$, but 
$$\Vert e_n  \Vert_{L^2(N_{h_n^{\delta}})}\sim C_1 h_n^{ \frac {d-1} 4 }, \quad \| e_n  \Vert_{L^2(N_{2h_n^{\delta} }  \setminus N_{h_n^{\delta}})}\leq  C_2 e^{-c h_n^{2\delta -1 }}, \quad n \to + \infty, $$
contradicting~\eqref{qm-N} since $2\delta -1 <0$.
\item No smoothness is assumed on the function $b\in L^\infty(\mathcal{M})$. Notice however that (contrarily to the results in~\cite{BuZu15}) the lower bound in~\eqref{below} is {\em not sufficient} (at least with our approach)  and we do need also the upper bound.

\item As will appear clearly in the proof, we could assume that $T$ is isometric to finitely many product manifolds, with possibly different constants $\gamma$, the final decay rate being given by the largest $\gamma$.
\end{enumerate}
 \end{rema}
The paper is organized as follows. We first show how to deduce from Theorem~\ref{quasi-mode} a resolvent estimate which according to previous works by Borichev-Tomilov imply Theorem~\ref{stabi}. Then we prove Theorem~\ref{quasi-mode} by elementary scaling and propagation arguments. 

  \section{From concentration to stabilization results  (Proof of Theorem~\ref{stabi})}\label{se.2}
  According to the works by Borichev-Tomilov~\cite{BoTo10}, stabilization results for the wave  equation are equivalent to resolvent estimates. As a consequence, to prove Theorem~\ref{stabi}, it is enough to prove (see~\cite[Proposition 1.5]{LeLe14})
  \begin{prop}\label{corsimple}
   We keep  the geometric assumptions in Theorem~\ref{stabi}.
Consider for $h>0$ the operator
 \begin{equation}\label{L}
 L_h= - h^2 \Delta_g -1 + i h b, \quad   b \in L^\infty(\mathcal{M}). 
\end{equation}
 Then there exist  $C>0, h_0>0$ such that for all $ 0<h\leq h_0$ 
$$\Vert \varphi  \Vert_{L^2(\mathcal{M})} \leq C h^{-1- \frac{ \gamma} {\gamma +1} } \Vert L_h \varphi \Vert_{L^2(\mathcal{M})},$$
  for all  $\varphi  \in H^2(\mathcal{M})$.
\end{prop}
\begin{proof}
We start with a simple {\em a priori estimate}.  Multiplying both sides of the the equation
\begin{equation}\label{equation}  
  (- h^2 \Delta_g -1 + i h b) \varphi =f.
  \end{equation}
  by $\overline{\varphi}$, integrating by parts on $\mathcal{M}$ and taking real and imaginary parts gives
\begin{align}\label{estapriori-1}
 h \int_\mathcal{M} b(m) \vert \varphi(m) \vert^2\, dv_g(m) &\leq \Vert \varphi \Vert_{L^2(\mathcal{M})}\Vert f  \Vert_{L^2(\mathcal{M})},\\
 \label{estapriori-2}
 h^2\int_\mathcal{M} g_m \big (\nabla_g \varphi(m), \overline{\nabla_g \varphi(m} \big) \, dv_g(m) &\leq  \Vert \varphi \Vert^2_{L^2(\mathcal{M})} +\Vert \varphi \Vert_{L^2(\mathcal{M})}\Vert f  \Vert_{L^2(\mathcal{M})}.
\end{align}  
Now, in the neighborhood $V$ of $T$ we use our isometry $\Theta$ and we set 
\begin{equation}\label{tilde}
 u(p,z) = \varphi( \Theta^{-1}(p,z)), \quad \widetilde{b}(p,z) = b(\Theta^{-1}(p,z)), \quad \widetilde{f}(p,z) = f( \Theta^{-1}(p,z)).  
 \end{equation}
  Then  from \eqref{equation} we obtain the equation on $M_1 \times B(0,1) $ 
$$ ( h^2 \Delta_{\widetilde{g}} +1) u =  i h \widetilde{b} u - \widetilde{f}.  $$

 We can therefore apply  Theorem~\ref{quasi-mode} and we obtain
 \begin{equation}\label{eq.1}
 \| u\|_{L^2(M_1 \times\{ |z| \leq h^{\delta}\})} \leq C  \| u\|_{L^2(M_1\times \{ h^{\delta}\leq  |z| \leq 2h^{\delta}\})} + C h^{2\delta -2} \|   i h \widetilde{b} u - \widetilde{f} \|_{L^2(M_1 \times \{  |z| \leq 2h^{\delta}\})}. 
\end{equation}
On the other hand, from~\eqref{estapriori-1}, and the lower bound in assumption~\eqref{below}, we deduce
\begin{equation}\label{eq.2}
\| u\|^2_{L^2(   M_1 \times \{h^{\delta}\leq  |z| \})} \leq C h^{-1-2\delta \gamma}\Vert \varphi \Vert_{L^2(\mathcal{M})}\Vert f  \Vert_{L^2(\mathcal{M})}
\end{equation}
while from the upperbound in assumption~\eqref{below}, we get
\begin{equation}\label{eq.3}
\begin{aligned}
\|  i h \widetilde{b} u \|^2_{L^2( \{ M_1 \times \{|z| \leq 2h^{\delta}\})} &\leq h^2 \big(\sup_{M_1 \times \{|z| \leq 2h^{\delta}\}}  |\widetilde{b}| \big)\| \widetilde{b}^{1/2} u\|^2_{L^2(M_1 \times \{|z| \leq 2h^{\delta}\}))}\\
& \leq C h^{1+2 \delta \gamma}\Vert \varphi \Vert_{L^2(\mathcal{M})}\Vert f  \Vert_{L^2(\mathcal{\mathcal{M}})}.
\end{aligned}
\end{equation}
Gathering~\eqref{eq.1},~\eqref{eq.2} and~\eqref{eq.3} we obtain,
\begin{equation}
\begin{aligned}
\|  u\|^2_{L^2(M_1 \times B(0,1))}\leq C  h^{-1-2\delta \gamma}&\Vert \varphi \Vert_{L^2(\mathcal{M})} \Vert f  \Vert_{L^2( \mathcal{M})}\\
& +  C h^{4\delta -4} \Bigl(h^{1+2 \delta \gamma}\Vert \varphi \Vert_{L^2(\mathcal{M})}\Vert f  \Vert_{L^2(\mathcal{M})}+ \| f\|_{L^2(\mathcal{M})}^2\Bigr).
\end{aligned}
\end{equation}
Optimizing with respect to $\delta$ leads to the choice $2\delta = \frac 1 {1+ \gamma}$, which gives
\begin{equation*} 
\|  u\|^2_{L^2(M_1 \times B(0,1))}\leq C  h^{-1- \frac \gamma{1+ \gamma}}\Vert \varphi \Vert_{L^2(\mathcal{\mathcal{M}})}\Vert f  \Vert_{L^2(\mathcal{M})} +  C h^{-2 - \frac{ 2\gamma}{1+ \gamma}} \| f\|_{L^2(\mathcal{M})}^2.
\end{equation*}
According to \eqref{tilde} this implies
\begin{equation}\label{borne}
 \Vert \varphi \Vert_{L^2(V)} \leq C  h^{-1- \frac \gamma{1+ \gamma}}\Vert \varphi \Vert_{L^2(\mathcal{\mathcal{M}})}\Vert f  \Vert_{L^2(\mathcal{M})} +  C h^{-2 - \frac{ 2\gamma}{1+ \gamma}} \| f\|_{L^2(\mathcal{M})}^2.
\end{equation}

 We can now conclude the proof of Proposition~\ref{corsimple} by contradiction. If ~\eqref{L} were not true, then there would exists sequences $\varphi_n\in H^2(\mathcal{M}), f_n\in L^2(\mathcal{M}), 0< h_n \rightarrow 0$ such that 
$$(- h_n^2 \Delta_g -1 + i h_n b) \varphi_n = f_n, \qquad \| \varphi_n\|_{L^2(\mathcal{M}) } > \frac n {h_n^{1+ \frac \gamma {1+ \gamma}}} \| f_n\|_{L^2(\mathcal{M})}.
$$ Dividing $\varphi_n$ by its $L^2$-norm, we deduce 
\begin{equation}\label{normalisation}
  \| \varphi_n \|_{L^2(\mathcal{M})} =1, \quad \| f_n \|_{L^2(\mathcal{M})}= o( h_n^{1+ \frac \gamma {1+ \gamma}}), \quad {n\rightarrow + \infty},
  \end{equation}
and from   \eqref{borne}  we get 
\begin{equation}\label{control}
 \lim_{n \to + \infty} \Vert \varphi_n \Vert_{L^2(V)} = 0.
\end{equation}
On the other hand, the sequence $(\varphi_n)$ is bounded in $L^2(\mathcal{M})$, and extracting a subsequence, we can assume that it has a semi-classical measure $\mu$ (see e.g. ~\cite[Th\'eor\`eme 2]{Bu97-1}). We recall that it means that for any symbol $a\in C^\infty_0 ( S^*\mathcal{M})$, 
$$ \lim_{n\rightarrow + \infty} \bigl( a(x, h_n D_x) \varphi_n, \varphi_n\bigr)_{L^2(\mathcal{M})} = \langle \mu, a \rangle.$$
Here, since we work locally, we quantize the symbols $a\in C^\infty_0( T^* \mathbb{R}^d)$ by taking first $\phi \in C^\infty _0( \mathbb{R}^d)$ equal to $1$ near the $x$-projection of the support of $a$ and 
$$ a(x, hD_x) u = \frac 1 {(2\pi h)^n} \int e^{\frac i h (x-y) \cdot \xi} a(x, \xi) \phi(y) u(y) dy d\xi.$$
It is classical that modulo $O(h^\infty)$ smoothing operators, the operator $a(x, hD_x)$ does not depend on the choice of $\phi$.

From~\eqref{estapriori-2},  the sequence $(\varphi_n)$ is $h_n$ oscillating and hence any such semi-classical defect measure has total mass $1= \lim_{n\rightarrow + \infty} \| \varphi_n \|_{L^2(\mathcal{M})}$ (see~\cite[Proposition 4]{Bu97-1}). From~\eqref{estapriori-1} and \eqref{normalisation} we also have (notice that $|b| \leq C |b|^{1/2}$) 
$$ (- h^2 \Delta -1) \varphi_n =  -ih_n b \varphi_n + f_n =o(h_n)_{L^2}, $$ 
and consequently (see~\cite[Proposition 4.4]{Bu02}) the measure $\mu$ is invariant by the bicharacteristic flow.  Since from~\eqref{estapriori-1} it is $0$ on $S^*\omega$, we deduce by propagation that it is also $0$ on $\mathcal{GC}$,  and hence  from~\eqref{control} it is identically null, since $S^*(\mathcal{M}) = \mathcal{T} \cup \mathcal{GC}.$ This gives the contradiction.   
\end{proof}
  \section{Concentration properties (Proof of Theorem~\ref{quasi-mode})}
Recall that we have $(M,g) = (M_1 \times M_2, g_1 \otimes g_2).$ The proof of Theorem~\ref{quasi-mode} follows, after taking scalar products with Laplace eigenfunctions in $M_1$,  from a rescaling argument and  standard (non trapping) resolvent estimates in $M_2$. When the metric $g_2$  is flat, the scaling argument is straightforward, while it requires a little care in the general case (see Lemma~\ref{lem.commut}).  

Let $B(q_0,r) \subset M_2$ be the ball (for the metric $d_2$) of radius $r>0$  centered at  $q_0.$
 
\begin{prop}\label{prop.3.1}
For any $\delta>0$, there exists $C>0, h_0>0$ such that for every $0<h\leq h_0$, every $\tau \in \mathbb{R}$, every  solution $U\in H^2(M_2), G\in L^2(M_2)$ of the equation on $M_2$
$$(- \Delta_{g_2}   - \tau )U  = G $$
we have the estimate
\begin{equation}\label{qm-N-bis}
\Vert U \Vert_{L^2(B(q_0,h^{\delta}))} \leq C\big( \Vert U  \Vert_{L^2(B(q_0,2h^{\delta})\setminus B(q_0,h^{\delta}))} + h^{2\delta } \Vert G \Vert_{L^2(B(q_0,2h^{\delta}))}\big).
 \end{equation}
\end{prop}
\subsection{Proof of Theorem \ref{quasi-mode} assuming Proposition \ref{prop.3.1}}
 Let $(e_n)$ be a sequence of eigenfunctions of the Laplace operator on $M_1$ with eigenvalues $- \lambda_n ^2$ forming  an $L^2(M_1)$ orthonormal basis.  For $\psi \in L^2(M)$, we set $\widehat{\psi}_n(q) = \big( \psi(\cdot,q), e_n \big)_{L^2(M_1)}.$ Then we have $\psi(p,q) = \sum_{n\in \mathbb N} \widehat{\psi}_n(q) e_n(p) $ and it is easy to see that with the notations in \eqref{Nbeta}, for $r>0$ 
 \begin{equation}\label{fourier}
\Vert \psi \Vert^2_{L^2(N_r)}  =  \Vert \psi \Vert^2_{L^2(M_1 \times B(q_0, r))} = \sum_{n\in \mathbb N} \Vert \widehat{\psi}_n \Vert^2_{L^2(B(q_0, r)}.
\end{equation}
  Now taking the scalar product of the equation \eqref{qm-N} with $e_n$ we see easily that $(-h^2 \Delta_{g_2} + h^2 \lambda_n^2 -1)\widehat{\psi}_n = \widehat{F}_n$ which can be rewritten as 
 $$ (-  \Delta_{g_2} - \tau )\widehat{\psi}_n = h^{-2}\widehat{F}_n, \quad \tau =    h^{-2} - \lambda_n^2.$$
 Applying Proposition \ref{prop.3.1} to this equation yields
 $$\Vert \widehat{\psi}_n \Vert^2_{L^2(B(q_0,h^{\delta}))} \leq C\big( \Vert \widehat{\psi}_n \Vert^2_{L^2(B(q_0,2h^{\delta})\setminus B(q_0,h^{\delta}))} + h^{4\delta } \Vert h^{-2}\widehat{F}_n \Vert^2_{L^2(B(q_0,2h^{\delta}))}\big).$$
 Taking the sum in $n$ and using \eqref{fourier} we obtain the estimate \eqref{qm-N}.
 
\subsection{Proof of Proposition \ref{prop.3.1}}

Since the problem is local near  $q_0$, after diffeomorphism we can work in a neighborhood of the origin in $\mathbb R_z^k$  and we may  assume that the new metric $g$ satisfies $g\arrowvert_{z=0} = \text{ Id}.$ Then we make the change of variables $z \mapsto x= \frac {z} { h^{\delta}}$ and we set $u(x) = U(h^{\delta } x)$, $F(x) = G( h^\delta x).$  We obtain the equation on $u$ 
$$  (-  \Delta_{g ^h}   - h^{2\delta} \tau )u  = h^{2\delta} F,$$ 
where $g ^h$ is the metric obtained by dilatation $ g ^h (x) = g (h^{\delta} x)$. The family $(g ^h)$ converges in $C^\infty$ topology to the flat metric $g_0 = \text{ Id}$. 
Proposition~\ref{prop.3.1} will follow easily from 
\begin{prop}\label{propag}
Consider a family  $(g_n) $ of metrics on $B(0,2) \subset \mathbb{R}^k$, which converges in Lipschitz topology to the   flat metric when $n\rightarrow +\infty$. Then  there exists $C>0, N_0 >0$ such that for every $n\geq N_0$,  $\tau \in \mathbb{R}$,  $u\in H^2(B(0,2)), f\in L^2(B(0,2))$ solutions of the equation on $B(0,2)$
$$(- \Delta_{g_n}   - \tau )u  = f $$
we have the estimate
\begin{equation}\label{qm-N-ter}
\Vert u \Vert_{L^2(B(0,1))} \leq C\big( \Vert u  \Vert_{L^2(B(0,2)\setminus B(0,1))} + \frac{ 1} { 1+ |\tau|^{1/2}} \Vert f  \Vert_{L^2(B(0,2))}\big)
 \end{equation}
 (notice that since $g_n$ converges to the flat metric the choice of the metric to define the $L^2$-norms above is of no importance).
\end{prop}
\begin{rema} Proposition~\ref{propag} is standard for the fixed metric $g_0 = \text{ Id}$ (see e.g.~\cite[Section 3]{Bu02}), as the annulus $\{x: 1<|x|<2\}$ controls geometrically the ball $ B(0,1)$. As a consequence, in the special case of~\cite{LeLe14} when $g = g_0$ (and hence $g_n$ is also the standard flat metric), the proof of Theorem~\ref{quasi-mode} is completed. In the general case, we only have to verify that the usual proof can handle the varying metric through a perturbation argument, which is precisely what we do below. It is worth noticing that the proof belows implies that the propagation estimates involved in exact controlability results which are known to hold for $C^2$ metrics, see~\cite{Bu97}, are actually {\em stable} by {\em small Lipschitz} perturbations of the metric.
\end{rema}  
For $r>0$ we shall set  $B_r = B(0,r) \subset \mathbb R^k.$ 

To prove Proposition~\ref{propag} we argue by contradiction. Otherwise, there would exist  sequences, $\sigma_n \to + \infty$, $(\tau_n)\subset \mathbb{R}$, $(u_n) \subset H^2(B_2), (f_n) \subset  L^2(B_2)$ such that
\begin{align}\label{eq-1}
 (- \Delta_{g_{\sigma_n}}   - \tau_n )u_n  &= f_n,\\ \label{eq-2}
1= \| u_{n} \|_{L^2( B_1)}   &> n \big( \Vert u_n \Vert_{L^2(B_2\setminus B_1)} + \frac{ 1} { 1+ |\tau_n|^{1/2}} \Vert f_n  \Vert_{L^2(B_2)}\big) 
 \end{align}
We now distinguish three cases 
\begin{itemize}
\item $\liminf_{n \rightarrow + \infty} {\tau_n} = - \infty$ (elliptic case)
\item $( \tau_n)_{N\in \mathbb{N}} $ bounded (low frequency case)
\item $\limsup_{n \rightarrow + \infty} {\tau_n}= + \infty$ (hyperbolic case)
\end{itemize}
In the first case, working with a subsequence we may assume that $\lim_{n \to + \infty} \tau_n = -\infty.$ Let $\zeta \in C^\infty_0 (B_2)$ equal to $1$ on $B_{3/2} $.  Multiplying~\eqref{eq-1} by $\zeta \overline{u}_n$, integrating by parts and taking the real part gives 
\begin{equation}
 \Bigl| \int (g_{\sigma_n} \big ( \nabla_{g_{\sigma_n}} u_n, \overline{ \nabla_{g_{\sigma_n}} (\zeta u_n))} \big) - \zeta \tau_n |u_n|^2 )  \,dv_{g_{\sigma_n}}   \Bigr|\leq \Vert u_n  \Vert_{L^2(B_2)}\Vert f_n  \Vert_{L^2(B_2)}
\end{equation}  
which implies (after another integration by parts)
\begin{equation}\label{estapriori}
\begin{aligned}
 \Bigl| \int \zeta g_{\sigma_n} \big ( \nabla_{g_{\sigma_n}} u_n, \overline{ \nabla_{g_{\sigma_n}}  u_n)} \big) &- \bigl(\tau_n \zeta + \frac{  \Delta_{g_{\sigma_n}} (\zeta)} 2 \bigr) |u_n|^2\, dv_{g_{\sigma_n}}    \Bigr|\\
 &\leq \Vert u_n \Vert_{L^2(B_2)}\Vert f_n  \Vert_{L^2(B_2)}= o(|\tau_n|^{1/2}), \quad {n\rightarrow + \infty}.
\end{aligned}
\end{equation}
Since $\Delta_{g_{\sigma_n}} \zeta$ is supported in $\{ 1\leq |x| \leq 2\}$ and $\| u_n\|_{L^2( 1< |x| <2)} = o(1)$, we deduce if $\tau_n \rightarrow - \infty $
$$ \lim_{n\rightarrow+ \infty} \int \zeta |u_n|^2 dx =0,$$ which contradicts \eqref{eq-2}.

In the second case (low frequency), we can assume (after extracting a subsequence) that $\tau_n \rightarrow \tau$ and~\eqref{estapriori} shows that the sequence $(u_n\arrowvert_{B_ {3/2}}) $ is bounded in $H^1(B_{3/2}  )$. Hence, (after taking a subsequence), we can assume that it converges weakly in $H^1( B _{3/2})$ (and hence strongly in $L^2(B_{3/2}$). Due to the convergence of the family of metrics, we get 
$$ -\Delta_{g_{\sigma_n}}u_{n} = - \Delta_0 u_n+ o(1)_{H^{-1}}, \quad  ( \Delta_0 = \sum_{i = 1}^k \partial_j^2),  $$   
and according to~\eqref{eq-2} this implies that the limit $u$ satisfies 
$$ (- \Delta_0 - \tau ) u =0 \text{ in } \mathcal{D}' (B_{3/2}), \qquad u\mid_{1< |z|<3/2}=0.$$
Uniqueness for solutions of second order elliptic operators implies that $u =0$ which is contradictory with the strong convergence of $(u_n) $ in $L^2( B_{3/2})$ and~\eqref{eq-2}.

Finally it remains to study the last case (hyperbolic). Taking a subsequence, we can assume $\tau_n \rightarrow + \infty.$ Moreover dividing both members of \eqref{eq-1} by $\tau_n$ we see that $u_n$ is solution of an equation of type $(P(x, \tau_n^{-\mez} D_x) - 1)u_n = \tau_n^{-1} f_n \to 0$ in $L^2(B(0,2)).$  The sequence $(u_n\arrowvert_{|x|<3/2})$ has   a semi-classical measure  $\nu$     with scale 
 $$\widetilde{h}_n = \tau_n ^{-1/2},$$
(see the end of Section~\ref{se.2} for a few fact about these measures). Notice that this new semi-classical parameter $\widetilde{h} _n$ has no relationship with the parameter $h$ in Theorem~\ref{quasi-mode} First of all multiplying both  sides of \eqref{estapriori} by $\widetilde{h}_n^2 = \tau_n^{-1}$ and using the fact that $\Vert u_n \Vert_{L^2(B_2)}$ is uniformly bounded we deduce that there exists $C>0$ such that
\begin{equation}\label{estapri2}
\widetilde{h}_n \Vert \nabla_x u_n \Vert_{L^2(B_{3/2})} \leq C, \quad \forall n \in \mathbb N. 
\end{equation}

Using again~\eqref{estapriori} shows that the sequence $u_n\arrowvert_{|x| <3/2}$ is $\widetilde{h}_n$-oscillatory (and hence the measure $\nu$ has total mass $1= \lim_{n\rightarrow +\infty} \| u_n \|_{L^2(B_ {3/2})}^2$). Now setting $D_n = \det ((g_{\sigma_n})_{ij})$ we can write
\begin{equation}\label{delta-g}
  \Delta_{g_{\sigma_n}}= \Delta_0 +   \sum_{i,j = 1}^k \partial_i \big\{   (g^{ij}_{{\sigma_n}} - \delta_{ij})  \partial_j \big\} + \frac{1}{2D_n}\sum_{i,j=1}^k  g_{{\sigma_n}}^{ij} (\partial_i   D_n)  \partial_j. 
\end{equation}
The only point of importance below will be that
\begin{equation}\label{cv-metr} 
  \lim_{n\rightarrow + \infty} \| g^{ij}_{\sigma_n}- \delta_{ij}\|_{W^{1, \infty} (B_2)}=0, \qquad  \lim_{n\rightarrow + \infty} \| D_n- 1\|_{W^{1, \infty} (B_2)}=0.
  \end{equation}
\begin{prop}\label{lem.propa}
The measure $\nu$ is supported in the set 
$\{ (x, \zeta): |\zeta\vert =1\}$ and is invariant by the bicharacteristic flow associated to the metric $g_0$:
$$ 2\xi\cdot \nabla_x \nu =0.$$
\end{prop}
The contradiction now follows since by \eqref{eq-2} we have $\| u_n\|_{L^2( 1<|z| <2)} \rightarrow 0$ which implies that $\nu \arrowvert_{1.1< |x| < 1.9} =0$ and by propagation that $\nu \arrowvert_{|x| < 3/2} =0$. It remains to prove Proposition~\ref{lem.propa}. 
\begin{proof}
We have for $a$ with compact support (in the $x$ variable) in $B(0,2)$,
\begin{equation}\label{ellipt}
\begin{aligned}
  & \bigl( a(x, \widetilde{h}_n D_x) ( \widetilde{h}_n ^2 \Delta_{g_{\sigma_n}} -1 ) u_n , u_n \bigr)_{L^2}  = (1) + (2 + (3),\\
  & (1) =  \bigl( a(x, \widetilde{h}_n D_x) ( \widetilde{h}_n ^2 \Delta_{0} - 1) u_n , u_n \bigr)_{L^2},\\
 &(2) = \sum_{i,j} \bigl(  (g^{ij}_ {\sigma_n}  - \delta_{ij})   \,  \widetilde{h}_n\partial_j  u_n ,   \widetilde{h}_n \partial_i  a^*(x, \widetilde{h}_n D_x)u_n \bigr)_{L^2},\\ 
&(3) =   \widetilde{h}_n \sum_{ij} \bigl( \frac{1}{2 D_n} (g^{ij}_ {\sigma_n} \partial_i D_n  \widetilde{h}_n \partial_{j} u_n,  a^*(x, \widetilde{h}_n D_x)u_n\bigr)_{L^2}.  
 \end{aligned}
 \end{equation}
 On one hand, using the symbolic calculus,  the term  $(1)$  tends to 
 $$\langle \nu, (|\zeta|^2 -1) a(x, \zeta) \rangle.$$ 
 
Now using~\eqref{estapri2}  and \eqref{cv-metr} we see easily that the terms $(2)$ and $(3)$ tend to zero when $n\rightarrow + \infty$. On the other hand,  the l.h.side in~\eqref{ellipt} is equal to 
 $$\widetilde{h}_n^2 \bigl( a(x, \widetilde{h}_n D_x) f_n , u_n \bigr)_{L^2}.$$
 and according to~\eqref{eq-2} tends to $0$. We deduce 
 $$ \forall a \in C^\infty_0 ( \mathbb{R}^{2k}, \langle \nu, (|\zeta|^2 -1) a(x, \zeta) \rangle\Rightarrow \text{ supp } (\nu ) \subset \{ (x, \zeta); |\zeta|^2 =1\}.$$
 To prove the second part in Proposition~\ref{lem.propa}, we shall use the following lemma
 \begin{lemm}\label{lem.commut}
 Let $a\in C^\infty_0 ( \mathbb{R}^{2k})$, and $b \in W^{1, \infty} (\mathbb{R}^{2k})$. Then 
 \begin{equation}\label{eq.commut}
 \| [ a(x, \widetilde{h}_n D_x), b ] \|_{\mathcal{L}( L^2)} \leq C \widetilde{h}_n \| \nabla_x b \| _{L^\infty}.
 \end{equation}
 \end{lemm}
 \begin{proof}
 The kernel of the operator $[ a(x, \widetilde{h}_n D_x), b ]$ is equal to (here $\phi\in C^\infty_0 (\mathbb{R}^k)$ is equal to $1$ on the $x$-projection of the support of $a$)
 $$ K(x, x') = \frac 1 { (2\pi \widetilde{h}_n )^k} \int_{\zeta \in \mathbb{R}^k} e^{\frac i {\widetilde{h}_n} \zeta\cdot  ( x- x') } a(x, \zeta) ( b(x)- b(x')) \phi(x')d \zeta,$$
 which is for $|x-x'|\leq \widetilde{h}_n$  (since the support of $a$ is compact) bounded by 
 \begin{equation}\label{schur1}
  C \widetilde{h}_n ^{-k}\| \nabla_x b\|_{L^\infty} |x-x'|,
  \end{equation}
  while for $|x-x'|\geq \widetilde{h}_n$
 we can integrate by parts using the identity 
 $$ \frac{ \widetilde{h}_n(x-x') } { i \vert x-x' \vert ^2} \cdot \nabla_\zeta (e^{\frac i {\widetilde{h}_n} \zeta \cdot ( x- x') }) = e^{\frac i {\widetilde{h}_n} \zeta \cdot ( x- x') },$$
 which gives 
 $$K(x,x') = \frac 1 { (2\pi \widetilde{h}_n )^k} \int_{\zeta \in \mathbb{R}^k} e^{\frac i {\widetilde{h}_n} \zeta \cdot( x- x') }  \Bigl( \widetilde{h}_n \frac{(x-x') \cdot \nabla_\zeta} { i \vert x-x' \vert^2}\Bigr)^N a(x, \zeta) ( b(x)- b(x') \phi(x')d \zeta, $$
 and hence gives the bound for any $N \in \mathbb{N}$,
  \begin{equation}\label{schur2}
   |K(x,x')| \leq \frac{ C_N \widetilde{h}_n ^{N-k}} { |x-x'|^{N-1}} \| \nabla_x b\|_{L^\infty}.
   \end{equation}
  It follows from \eqref{schur1} and \eqref{schur2} that
  $$\int_{\xR^k} \vert K(x,x') \vert dx + \int_{\xR^k} \vert K(x,x') \vert dx' \leq C \widetilde{h}_n  \| \nabla_x b\|_{L^\infty}.$$ 
  Then Lemma~\ref{lem.commut} follows from Schur's lemma. 
 \end{proof}
 Denoting by $[A, B]$ the commutator of the operators $A$ and $B$ let us set
 $$  C =   \frac {i} {\widetilde{h}_n} \bigl( \big [ a(x, \widetilde{h}_n D_x),  ( \widetilde{h}_n ^2 \Delta_{g_{\sigma_n}} -1) \big] u_n , u_n \bigr)_{L^2 }. $$
 Then we can write using  \eqref{eq-1} and  \eqref{delta-g},
 \begin{equation}\label{hyper}
 \begin{aligned}
  C &= \frac i { \widetilde{h}_n} \bigl(\Big[ a(x, \widetilde{h}_n D_x), \widetilde{h}_n ^2f_n \big], u_n \bigr) _{L^2} = (1) + (2) + (3)\\
      (1) &= \frac {i} {\widetilde{h}_n} \bigl(\big [a(x, \widetilde{h}_n D_x),  ( \widetilde{h}_n ^2 \Delta_{0} - 1)\big ] u_n , u_n \bigr)_{L^2 },\\
  (2) &= \frac {i} {\widetilde{h}_n}  \sum_{j,l = 1}^k \big( \big[a(x, \widetilde{h}_n D_x), \widetilde{h}_n\partial_j \big(  (g^{j l }_{{\sigma_n}} - \delta_{jl})  \widetilde{h}_n\partial_l \big) \big]u_n, u_n \big)_{L^2},\\ 
   (3) &=\frac {i} {\widetilde{h}_n}  \sum_{j,l =1}^k  \widetilde{h}_n\big( \big[a(x, \widetilde{h}_n D_x), \frac{1}{2D_n}  g_{\sigma_n}^{jl} (\partial_j   D_n) \widetilde{h}_n \partial_l\big] u_n, u_n \big)_{L^2}.\\
    \end{aligned} 
 \end{equation}
By symbolic calculus, the   term $(1)$  is modulo an $\mathcal{O}( \widetilde{h}_n)$ term equal to 
$$ \bigl( \text{Op}(\{a(x, \zeta) ,  |\zeta|^2\}) u_n, u_n \bigr)_{L^2  },$$   
     where $\{,\}$ denotes the Poisson bracket, and hence  tends to 
$$ \langle \nu, \{a(x, \zeta) ,  |\zeta|^2\}\rangle =   \langle 2\zeta \cdot \nabla_x \nu,  a\rangle.$$
Let us look to   $(2)$.   Each term in the   sum   can be bounded by 
\begin{multline*}
  \frac{1}{\widetilde{h }_n} \vert\big( \big[a(x, \widetilde{h}_n D_x), \widetilde{h}_n\partial_j\big] (g^{j l }_{{\sigma_n}} - \delta_{jl})  \widetilde{h}_n\partial_l u_n, u_n \big)_{L^2}\vert \\
  +\frac{1}{\widetilde{h }_n}\big(\big[a(x, \widetilde{h}_n D_x), g^{j l }_{{\sigma_n}} - \delta_{jl})\big] \widetilde{h}_n\partial_l u_n, \widetilde{h}_n\partial_j u_n\big) \vert\\
  +  \frac{1}{\widetilde{h }_n} \vert \big(( g^{j l }_{{\sigma_n}} - \delta_{jl}) \big[a(x, \widetilde{h}_n D_x), \widetilde{h}_n\partial_l\big]u_n,  \widetilde{h}_n\partial_j u_n\big) \vert.
  \end{multline*}
By the semiclassical symbolic calculus and Lemma \ref{lem.commut} the norms in $\mathcal{L}(L^2)$ of the operators $\big[a(x, \widetilde{h}_n D_x), \widetilde{h}_n\partial_j\big]$ and  $\big[a(x, \widetilde{h}_n D_x), g^{j l }_{{\sigma_n}} - \delta_{jl})\big]$ are bounded respectively by  $C \widetilde{h}_n $ and $C \widetilde{h}_n \Vert \nabla_x g^{j l }_{{\sigma_n}} \Vert_{L^\infty}$ where $C$ is independent of $n$. Therefore using \eqref{cv-metr} and \eqref{estapri2} we deduce that $(2)$ tends to zero when $n$ goes to $+ \infty.$

Unfolding the commutator and using \eqref{eq-2}, \eqref{estapri2} we see that the third  term in~\eqref{hyper} is a finite sum of terms which are bounded by
 $C \Vert \partial_j   D_n\Vert_{L^\infty}. $  We deduce from \eqref{cv-metr} that $(3)$ tends to zero when $n$ goes to $+ \infty.$

Now, opening the commutator we see that the r.h.s. in the first equation in ~\eqref{hyper} is  equal to 
  $$ \frac i { \widetilde{h}_n} \bigl( a(x, \widetilde{h}_n D_x) \widetilde{h}_n ^2f_n, u_n \bigr) _{L^2} -\frac i { \widetilde{h}_n} \bigl( \widetilde{h}_n ^2 f_n, a^*(x, \widetilde{h}_n D_x)u_n \bigr) _{L^2}.$$ 
  These terms are bounded by $C \widetilde{h}_n \Vert f_n \Vert_{L^2(B_2)}\Vert u_n \Vert_{L^2(B_2)}$ and tend  to zero when goes to $+ \infty$ since, according to~\eqref{eq-2},  $ \Vert u_n \Vert_{L^2(B_2)}$ is uniformly bounded and  $\Vert f_n \Vert_{L^2(B_2)} = o(\tau^\mez_n)=o(\widetilde{h}_n^{-1})$ . This ends the proof of Lemma~\ref{lem.propa}, and hence of Proposition~\ref{propag}
\end{proof}

\def\cprime{$'$} \def\cprime{$'$}

\end{document}